\documentclass{amsart}
\usepackage{graphicx, hyperref}
\pdfoutput=1

\newcommand{\ceil}[1]{\left\lceil#1\right\rceil}
\newcommand{\floor}[1]{\left\lfloor#1\right\rfloor}

\newtheorem{theorem}{Theorem}[section]
\newtheorem{lemma}[theorem]{Lemma}

\newtheorem{question}[theorem]{Question}

\title{A note on list-coloring powers of graphs}
\author[N. Kosar, S. Petrickova, B. Reiniger, E. Yeager]{Nicholas Kosar\and Sarka Petrickova\and Benjamin Reiniger\and Elyse Yeager}
\thanks{kosar2@illinois.edu, petrckv2@illinois.edu, reinige1@illinois.edu, yeager2@illinois.edu\\Mathematics Dept., University of Illinois, Urbana-Champaign\\The authors acknowledge support from National Science Foundation grant DMS 08-­‐38434 ``EMSW21-­‐MCTP: Research Experience for Graduate Students''.}

\begin{document}
\begin{abstract}
Recently, Kim and Park have found an infinite family of graphs whose squares are not chromatic-choosable.  Xuding Zhu asked whether there is some $k$ such that all $k$th power graphs are chromatic-choosable.  We answer this question in the negative: we show that there is a positive constant $c$ such that for any $k$ there is a family of graphs $G$ with $\chi(G^k)$ unbounded and $\chi_{\ell}(G^k)\geq c \chi(G^k) \log \chi(G^k)$.  We also provide an upper bound, $\chi_{\ell}(G^k)<\chi(G^k)^3$ for $k>1$.
\end{abstract}

\maketitle

\section{Introduction}
The \emph{list-chromatic number} of a graph $G$, denoted $\chi_{\ell}(G)$, is the least $k$ such that for any assignment of lists of size $k$ to the vertices of $G$, there is a proper coloring of $V(G)$ where the color at each vertex is in that vertex's list.  A graph is said to be \emph{chromatic-choosable} if $\chi_{\ell}(G)=\chi(G)$.  The \emph{$k$th power} of a graph $G$, denoted by $G^k$, is the graph on the same vertex set as $G$ such that $uv$ is an edge if and only if the distance from $u$ to $v$ in $G$ is at most $k$.

The List Total Coloring Conjecture (LTCC) asserts that $\chi_{\ell}(T(G))=\chi(T(G))$ for every graph $G$, where $T(G)$ is the total graph of $G$.  The List Square Coloring Conjecture (LSCC) was introduced in \cite{KW}, as it would imply the LTCC.  The LSCC asserts that squares of graphs are chromatic-choosable.  However, the LSCC was recently disproved by Kim and Park \cite{KP}, who constructed a family of graphs $G$ with $\chi(G^2)$ unbounded and $\chi_{\ell}(G^2)\geq c \chi(G^2) \log \chi(G^2)$.  Xuding Zhu asked whether there is any $k$ such that all $k$th powers are chromatic-choosable \cite{Z}.  In this note we give a negative answer to Zhu's question, with a bound on $\chi_{\ell}(G^k)$ that matches that of Kim and Park for $k=2$.

\newtheorem*{mainthm}{Theorem \ref{mainthm}}
\begin{mainthm}
There is a positive constant $c$ such that for every $k\in\mathbb{N}$, there is an infinite family of graphs $G$ with $\chi(G^k)$ unbounded such that 
\[ \chi_{\ell}(G^k) \geq c \chi(G^k) \log \chi(G^k) . \]
\end{mainthm}

While preparing this note, it has come to our attention that Kim, Kwon, and Park have arrived at a similar result \cite{KKP}.  They have found, for each $k$, an infinite family of graphs $G$ whose $k$th powers satisfy $\chi_{\ell}(G^k) \geq \frac{10}{9}\chi(G^k)-1$.

Let $f_k(m)=\max\{\chi_{\ell}(G^k): \chi(G^k)=m\}$.  Then Theorem \ref{mainthm} says that $f_k(m)\geq c m \log m$.  Kwon (see \cite{N}) observed that $f_2(m)<m^2$.  We extend this observation to larger $k$ in section \ref{sec:upperbd}.
\newtheorem*{thm:upperbd}{Theorem \ref{thm:upperbd}}
\begin{thm:upperbd}
Let $k>1$.  If $k$ is even, then $f_k(m)<m^2$.  If $k$ is odd, then $f_k(m)<m^3$.
\end{thm:upperbd}
\begin{question}
What is the correct order of magnitude of $f_k(m)$?  Does it depend on $k$?
\end{question}

\section{Construction}\label{sec:construction}

The example of Kim and Park \cite{KP} for $k=2$ is based on complete sets of mutually orthogonal latin squares.  We will use this structure to find examples for all $k$, but we find the language of affine planes to be more convenient.

Take an affine plane $(\mathcal{P},\mathcal{L})$ on $n^2$ points.  Let $\{L_0 , L_1 , \dotsc , L_n\}$ be the decomposition of $\mathcal{L}$ into parallel classes.  Recall that we call the elements of $\mathcal{P}$ the \emph{points} and the elements of $\mathcal{L}$ the \emph{lines} of the plane, and that we have the following properties (see for instance \cite{CRC}):
\begin{itemize}
\item Each line is a set of $n$ points.
\item For each pair of points, there is a unique line containing them.
\item Two lines in the same parallel class do not intersect.
\item Two lines in different parallel classes intersect in exactly one point.
\item Such a plane exists whenever $n$ is a (positive) power of a prime.
\end{itemize}

Form the bipartite graph $H$ with parts $\mathcal{P}$ and $B=\mathcal{L} - L_0$, with $p\ell\in E(H)$ if and only if $p\in\ell$.  Let $a_1,\dotsc,a_n$ denote the lines of $L_0$.  Consider the refinement $\mathcal{V'}$ of the bipartition of $H$ obtained by partitioning $\mathcal{P}$ into $a_1, \dotsc, a_n$ and $B$ into $L_1, \dotsc, L_n$.  Note that $H[a_i, L_j]$ is a matching for each $i$ and $j$.  In Figure \ref{figH}, the graph $H$ is shown with $n=3$.  Edges are drawn differently according to which parallel class their line-endpoint belongs to, and the parts of $\mathcal{V}'$ are indicated.

\begin{figure}
\centering
\includegraphics{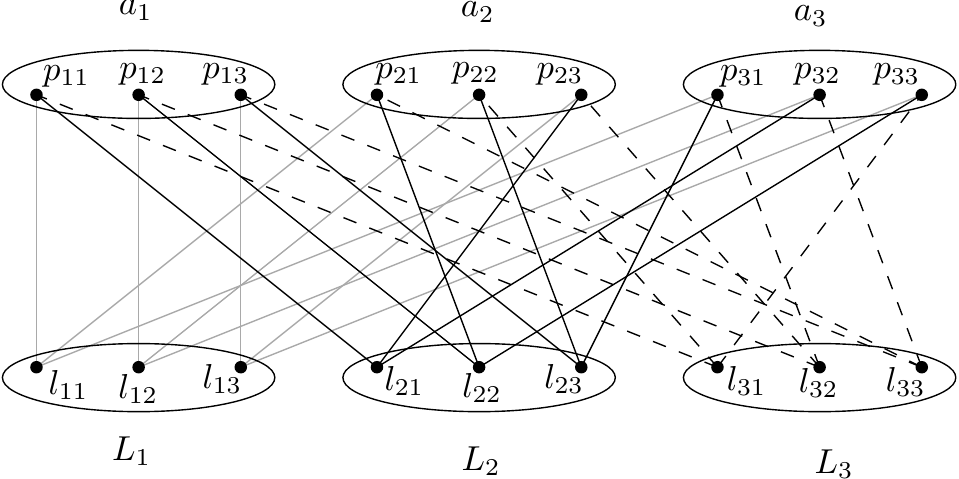}
\label{figH}
\caption{The graph $H$, here with $n=3$.}
\end{figure}

Let $k\geq2$.  Subdivide the edges of $H$ into paths of different lengths: edges incident to $L_1$ are subdivided into paths of length $k$, while edges not incident to $L_1$ are subdivided into paths of length $k+1$.  For an edge $p\ell\in E(H)$, denote the vertices along the subdivision path as $p=(p\ell)_0, (p\ell)_1, (p\ell)_2, \dotsc$.  If $\ell\in L_1$, then $(p\ell)_k=\ell$, and if $\ell\notin L_1$, then $(p\ell)_{k+1}=\ell$.  For a vertex $(p\ell)_i$, say its \emph{level} is $i$, its \emph{point} is $p$, and its \emph{line} is $\ell$ (levels are well-defined, and points and lines of vertices of degree 2 are well-defined).  Form the graph $G$ by, for each $\ell\in\bigcup_{2\leq i\leq n} L_i$, adding edges to make the neighborhood of $\ell$ a clique and then deleting $\ell$.  For each $i,j\in[n]$ and $m\in\{0,\dotsc,k\}$, let $V_{i,j,m} = \{ (p\ell)_m: p\ell\in E(H), p\in a_i, \ell\in L_j\}$; then $\{V_{i,j,m}: i,j\in[n], m\in\{0,\dotsc,k\}\}$ is a partition of $V(G)$ into sets of size $n$, which we call $\mathcal{V}$.  In Figure $\ref{figG}$, the graph $G$ is shown.  Again we use $n=3$, and here the parts of $\mathcal{V}$ are indicated.

\begin{figure}
\centering
\includegraphics{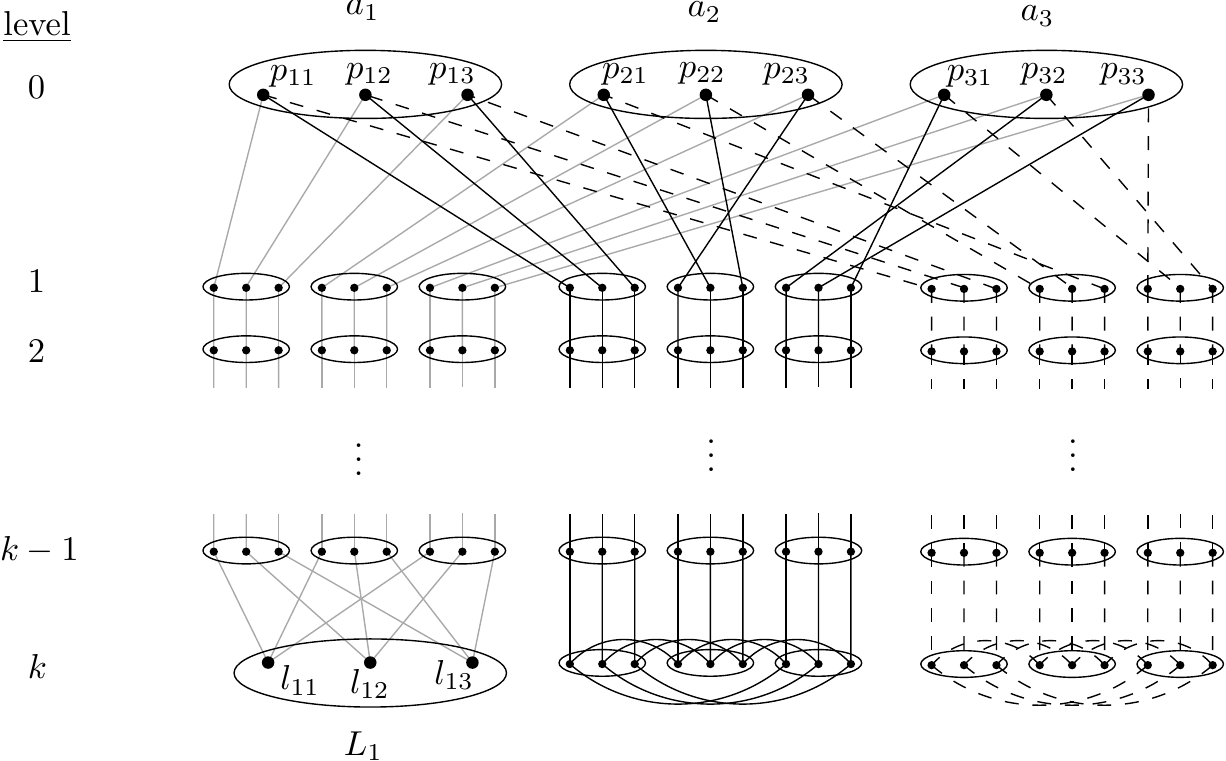}
\label{figG}
\caption{The graph $G$ when $n=3$.}
\end{figure}

\section{Proof of Theorem \ref{mainthm}}

\begin{lemma}
$G^{4k}$ is multipartite with partition $\mathcal{V}$.
\end{lemma}

\begin{proof}
Let $p$ and $q$ be two points in some $a_i$. Any path from $p$ to $q$ must start by increasing levels, arriving at $(p\ell)_{k}$. If $\ell \notin L_1$, then the path must move from $(p\ell)_{k}$ to $(p'\ell)_{k}$ for some $p'$ not on $a_i$. Continuing along the path to level $0$, we arrive at $p'$. Since $p'$ is not on $a_i$, $p'$ and $q$ are on a common line $\ell'\in \bigcup_{i=1}^n L_i$. If $\ell'\in L_1$, the shortest path from $p'$ to $q$ is to increase levels to $\ell'$ and decrease levels to $q$. If $\ell'\in \bigcup_{i=2}^n L_i$, the shortest path from $p'$ to $q$ is to increase levels to $(p'\ell')_{k}$, move over to $(q\ell')_{k}$, and then decrease levels to $q$. Notice, if $p$ and $p'$ are on a common line in $L_1$, $p'$ and $q$ cannot be on a common line in $L_1$ because then $p$ and $q$ would be on a common line in $L_1$. Thus, the path uses at least $3$ vertices in level $k$, and so has length at least $4k+1$.

Let $\ell_1, \ell_2 \in L_1$. Any path would have to have both ends decrease to level $0$. If both $\ell_1$ and $\ell_2$ connect to points in some $a_i$, then since these vertices are a distance at least $4k+1$ apart, the path between $\ell_1$ and $\ell_2$ would have length at least $4k+1$. Otherwise, the paths from $\ell_1$ and $\ell_2$ arrive at points on different lines in $L_0$, say $p$ and $q$, respectively. These two points are on a common line not in $L_0$ or $L_1$, say $\ell$. The shortest path between $p$ and $q$ is to go from $p$ to $(p\ell)_{k}$, over to $(q\ell)_{k}$, and finally to $q$. However, this results in a path between $\ell_1$ and $\ell_2$ of length at least $4k+1$.

Let $(p\ell_1)_{k}, (q\ell_2)_{k}$ be two vertices in the same part other than $L_1$; that is, $p, q$ are both on some $a_i$ and $\ell_1, \ell_2$ are two lines in the same parallel line class. If a path joining them starts by decreasing levels from both ends to level $0$, that is connects $(p\ell_1)_{k}$ to $p$ and $(q\ell_2)_{k}$ to $q$, then since $p$ and $q$ are a distance at least $4k+1$ apart, the path between $(p\ell_1)_{k}$ and $(q\ell_2)_{k}$ would have length at least $4k+1$. Otherwise, at least one of $(p\ell_1)_{k}$ or $(q\ell_2)_{k}$ must first go to $(p'\ell_1)_{k}$ or $(q'\ell_2)_{k}$. Without loss of generality connect $(p\ell_1)_{k}$ to $(p'\ell_1)_{k}$. Now, any path must connect $(p'\ell_1)_{k}$ to $p'$ and $(q\ell_2)_{k}$ to $q$. These are on a common line not in $L_0$, however, increasing levels from each of $p'$ and $q$ to level $k$ results in a total of at least $4k+1$ steps.

Now consider two degree-two vertices in the same part.  Any path joining them has ends that either increase or decrease levels from the endpoint.  If the path increases levels from both ends or decreases levels from both ends, then we arrive at different vertices in the same level 0 or level $k$ part. Since the rest of the path must have length at least $4k+1$, the total path must have length at least $4k+1$. Otherwise, one end increases levels and the other decreases levels. The resulting point, $p$, is not on the resulting line, $\ell$. The path must next increase levels from $p$ to a line. If this line is in the same parallel line class as $\ell$, then the resultant path has length over $4k+1$. Otherwise, since this line is not in the same class as $\ell$, these two lines share a common point. The shortest completion of the path is through this point. However, since at least one of these lines is not in $L_1$, the path must contain at least $3$ vertices in level $k$. Thus, the path has length at least $4k+1$.
\end{proof}

\begin{lemma}
The subgraph of $G^{4k}$ induced by the vertices in levels $0$ through $k-1$ is complete multipartite with partition $\mathcal{V}$ restricted to those levels.
\end{lemma}

\begin{proof}
Consider two points $p, q$ on different lines in $L_0$. They are on a common line $\ell\in \bigcup_{i=1}^n L_i$.  If $\ell\in L_1$, connect $p$ to $\ell$ then $\ell$ to $q$. If $\ell\notin L_1$, connect $p$ to $(p\ell)_k$ to $(q\ell)_k$ to $q$.  In each case the path has length at most $2k+1<4k$.

Consider two vertices in different parts at level $i$, $1\leq i\leq k-1$.  Either their points are on different lines in $L_0$ or their lines are from different parallel classes. If their points are from different lines in $L_0$, go to these points. These points share a common line not in $L_0$. Connect via the path between this line. This takes at most $2i+2k+1\leq4k-1$ steps. If their lines are from different parallel classes, increase levels to level $k$.  These two lines share a common point.  By, if necessary, first changing vertices at level $k$, connecting through this point, we get a path of length at most $2(k-i)+2+2k=4k-2i+2\leq4k$.

Finally, consider two vertices in levels $i$ and $j$, $0\leq i<j<k$.  Start a path joining them by decreasing levels from the lower-level vertex, and increasing levels from the larger-level vertex.  Let the point we arrive at from decreasing the lower-level vertex be $p$. If the increasing from the larger-level vertex takes us to a line in $L_1$, we can connect from this line to a point on a different line of $L_0$ than $p$, say $q$. Now $p$ and $q$ are on a common line not in $L_0$. Connecting through this gives us a path of length at most $k-1+k+2k+1=4k$.  If instead the increasing from the larger-level vertex takes us to a vertex of the form $(q\ell)_k$, $\ell\notin L_1$, then let $\ell'$ be the line through $p$ in $L_1$.  Now $\ell$ and $\ell'$ intersect at a point, say $q'$. We can complete the path by going from $(q\ell)_k$ to $(q'\ell)_k$ to $q'$ to $\ell'$ to $p$. This takes a total of at most $k-1+1+3k=4k$ steps.
\end{proof}

We will use the following result of Alon.
\begin{lemma}\cite{A}
\label{AlonLemma}
Let  $K_{r*s}$ denotes the complete $r$-partite graph with each part of size $s$.  There are two constants, $d_1$ and $d_2$, such that 
\[ d_1 r \log s \leq \chi_{\ell}(K_{r*s}) \leq d_2 r \log s. \]
\end{lemma}

Everything is now in place to complete the proof.
\begin{theorem}
\label{mainthm}
There is a positive constant $c$ such that for every $k\in\mathbb{N}$, there is an infinite family of graphs $G$ with $\chi(G^k)$ unbounded such that 
\[ \chi_{\ell}(G^k) \geq c \chi(G^k) \log \chi(G^k) . \]
\end{theorem}

\begin{proof}
Since $G^{4k}$ is multipartite on $k n^2+1$ parts, $\chi(G^{4k}) \leq k n^2+1$, and so $n \geq \sqrt{(\chi(G^{4k})-1)/k}$.

Since $G^{4k}$ contains a complete multipartite subgraph with $(k-1) n^2$ parts of size $n$, we have from Lemma \ref{AlonLemma} that
\begin{align*}
\chi_{\ell}(G^{4k}) &\geq d_1 (k-1) n^2 \log n \\
  &\geq d_1 \frac{k-1}{k} \left(\chi(G^{4k})-1\right) \log \sqrt{\frac{\chi(G^{4k})-1}{k}} \\
  &= \frac{d_1}{2} \frac{k-1}{k} \left(\chi(G^{4k})-1\right) \left( \log (\chi(G^{4k})-1) - \log k \right) \\
  &\geq \frac{d_1}{4} \left(\chi(G^{4k})-1\right) \left( \log (\chi(G^{4k})-1) - \log k \right).
\end{align*}
Taking $n$ large enough makes $\chi(G^{4k})$ as large as we like, and so by taking a constant $c$ just smaller than $d_1/4$ and taking $n$ sufficiently large, we obtain
\[ \chi_{\ell}(G^{4k}) \geq c \chi(G^{4k}) \log \chi(G^{4k}). \]
The family $\{G^4\}$ is an infinite family of graphs whose $k$th powers have the desired properties.
\end{proof}

\section{Upper bound}\label{sec:upperbd}
We now provide an upper bound on $\chi_{\ell}(G^k)$ in terms of $\chi(G^k)$.  Recall that $f_k(m)=\max\{\chi_{\ell}(G^k): \chi(G^k)=m\}$.
\begin{theorem}\label{thm:upperbd}
Let $k>1$.  If $k$ is even, then $f_k(m)<m^2$.  If $k$ is odd, then $f_k(m)<m^3$.
\end{theorem}
When $k$ is even, this follows from Kwon's observation (see \cite{N}) that it holds for $k=2$.  When $k$ is odd, we generalize the argument and prove the following.
\begin{theorem}\label{deltachiUB}
Let $k\geq3$, $k$ odd.  Then for any $G$, $\chi_{\ell}(G^k) \leq \Delta(G)\, \chi(G^k)^2$.
\end{theorem}
Theorem \ref{thm:upperbd} follows by noting that $\Delta(G)<\omega(G^k)\leq\chi(G^k)$ when $k>1$.
\begin{proof}[Proof of Theorem \ref{deltachiUB}]
Let $x$ be a vertex with maximum degree in $G^k$.  Let $A$ be the set of vertices at distance $\ceil{k/2}$ from $x$ in $G$.  Let $B(v,r)$ denote the ball of radius $r$ centered at $v$ in $G$.  Note that $\Delta(G^k) = \max\{|B(v,k)|-1: v\in V(G)\}$ and $\omega(G^k)\geq \max\{|B(v, \floor{k/2})|: v\in V(G)\}$.

Since $k$ is odd (and bigger than 1), we have
\begin{equation} \label{containment} 
B(x,k)\setminus B(x,\floor{k/2}) \subseteq\bigcup_{y\in A} B(y,\floor{k/2}).
\end{equation}
Let $S$ be the set of vertices at distance $\floor{k/2}$ from $x$ in $G$.  Then $S$ is a clique in $G^k$, so $|S|\leq \omega(G^k)$.  Also, $A$ is contained in the neighborhood of $S$, and each vertex in $S$ also has at least one neighbor outside of $A$ (closer to $x$).  Hence $|A|\leq (\Delta(G)-1)|S|\leq (\Delta(G)-1)\omega(G^k)$.  So
\begin{align*}
 \chi_{\ell}(G^k) &\leq 1+\Delta(G^k) & \text{(degeneracy)} \\
 &= |B(x,k)| \\
 &\leq |B(x,\floor{k/2})| + \sum_{y\in A}|B(y,\floor{k/2})| & \text{(equation (\ref{containment}))}\\
 &\leq (1+|A|) \max_{v\in V(G)} |B(v,\floor{k/2})| &\text{(bounding terms in sum)}\\
 &\leq \left(1+(\Delta(G)-1)\omega(G^k)\right) \omega(G^k) \\
 &\leq \Delta(G) \, \omega(G^k)^2 \\
 &\leq \Delta(G) \, \chi(G^k)^2. & \qedhere
\end{align*}
\end{proof}

\section{Remarks}
Using constructions similar to that of section \ref{sec:construction}, we have found infinite families of graphs $G$ whose $k$th powers are complete multipartite on roughly $kn^2/4$ parts each of size $n$, but only when $k\not\equiv0\mod4$.  The construction presented here is messier and does not yield complete multipartite powers, but it proves the theorem for all values of $k$ simultaneously.

The authors would like to thank Douglas West for bringing this problem to our attention and helping to improve the exposition of this note, Alexandr Kostochka and Jonathan Noel for their encouragement and helpful comments, and the anonymous referees for their helpful advice.

\bibliographystyle{plain}
\bibliography{powerlistchr}

\begin{thebibliography}{1}

\bibitem{A}
Noga Alon.
\newblock Choice numbers of graphs: a probabilistic approach.
\newblock {\em Combin. Probab. Comput.}, 1(2):107--114, 1992.

\bibitem{CRC}
Charles~J. Colbourn and Jeffrey~H. Dinitz, editors.
\newblock {\em The {CRC} handbook of combinatorial designs}.
\newblock CRC Press Series on Discrete Mathematics and its Applications. CRC
  Press, Boca Raton, FL, 1996.

\bibitem{KKP}
Seog-Jin Kim, Young~Soo Kwon, and Boram Park.
\newblock Chromatic-choosability of the power of graphs.
\newblock arXiv:1309.0888.

\bibitem{KP}
Seog-Jin Kim and Boram Park.
\newblock Counterexamples to the list square conjecture.
\newblock submitted. arXiv:1305.2566.

\bibitem{KW}
Alexandr~V. Kostochka and Douglas~R. Woodall.
\newblock Choosability conjectures and multicircuits.
\newblock {\em Discrete Math.}, 240(1-3):123--143, 2001.

\bibitem{N}
Jonathan~A. Noel.
\newblock Choosability of graph powers, August 2013.
\newblock http://www.openproblemgarden.org/op/choosability\_of\_graph\_powers.

\bibitem{Z}
Xuding Zhu.
\newblock personal communication.

\end{thebibliography}

\end{document}